\documentclass[notitlepage,leqno,9pt]{amsart}
\usepackage{amsfonts,amssymb,cite}
\usepackage[latin1]{inputenc}
\usepackage{enumerate}
\usepackage{amssymb}
\catcode`\@=11 \@addtoreset{equation}{section}

\catcode`\@=12
\usepackage{latexsym}
\usepackage{textcomp}
\usepackage{amsmath}
\usepackage{amsfonts}
\usepackage{mathrsfs}

\usepackage{color}
\usepackage{amsmath, amsthm, amssymb}
\usepackage{amsfonts}
\usepackage[dvips]{epsfig}
\usepackage{graphicx}
\usepackage{caption}
\usepackage{subcaption}
\usepackage[english]{babel}
\usepackage{hyperref}

\usepackage{tikz}
\usepackage{rotating}
\usepackage{cite}
\usepackage{amscd}
\usepackage{color}
\usepackage{bm}
\usepackage{enumerate}
\usepackage{verbatim}
\usepackage{hyperref}
\usepackage{amstext}
\usepackage{latexsym}
\theoremstyle{plain}
\newtheorem{theorem}{Theorem}[section]

\newtheorem{proposition}[theorem]{Proposition}
\newtheorem{lemma}[theorem]{Lemma}

\numberwithin{theorem}{section}
\numberwithin{equation}{section}

\newcommand{\average}{{\mathchoice {\kern1ex\vcenter{\hrule height.4pt
width 6pt depth0pt} \kern-9.7pt} {\kern1ex\vcenter{\hrule
height.4pt width 4.3pt depth0pt} \kern-7pt} {} {} }}

\def\R{\mathbb{R}}

\renewcommand{\a }{\alpha }
\renewcommand{\b }{\beta }
\renewcommand{\d}{\delta }

\newcommand{\D }{\Delta }
\newcommand{\tr }{\hbox{ tr } }
\newcommand{\e }{\varepsilon }

\newcommand{\G }{\Gamma}
\renewcommand{\l }{\lambda }

\newcommand{\n }{\nabla }
\newcommand{\vp }{\varphi }
\renewcommand{\phi}{\varphi}

\newcommand{\s }{\sigma }

\renewcommand{\th }{\theta }

\renewcommand{\O }{\Omega }

\newcommand{\ov}{\overline}

\newcommand{\be}{\begin{equation}}
\newcommand{\ee}{\end{equation}}

\newcommand{\de}{\partial}

\newcommand{\ti}{\widetilde}

\renewcommand{\k}{\kappa}

\newcommand{\calC }{\mathcal{C}}

\newcommand{\calD }{\mathcal{D}}

\newcommand{\N}{\mathbb{N}}

\newcommand{\cD}{{\mathcal D}}

\newcommand{\dist}{{\rm dist}}

\newcommand{\B}{{Q}}

\renewcommand{\epsilon}{\varepsilon}

\begin{document}
 
\title[Hardy-Sobolev inequality with higher dimensional singularity]
{  Hardy-Sobolev inequality with higher dimensional singularity}
%

\author{El Hadji Abdoulaye Thiam}
\address{E. H. A. T.: African Institute for Mathematical Sciences in Senegal, KM 2, Route de
Joal, B.P. 14 18. Mbour, Senegal. }
\email{elhadji@aims-senegal.org}
\begin{abstract}
For $N\geq 4$, we let $\O$ to be a smooth  bounded domain of $\R^N$, $\G$ a smooth closed submanifold of $\O$ of dimension $k$ with $1\leq k \leq N-2$ and $h$ a continuous function defined on $\O$. We denote by $\rho_\G\left(\cdot\right):=\dist_g\left(\cdot, \G\right)$ the distance function to $\G$. For $\s\in (0,2)$, we study existence of positive solutions $u \in H^1_0\left(\O\right)$ to the nonlinear equation
$$
-\D u+h u=\rho_\G^{-\s} u^{2^*(\s)-1} \qquad \textrm{in } \O,
$$
where $2^*(\s):=\frac{2(N-\s)}{N-2}$ is the critical Hardy-Sobolev exponent. In particular, we provide existence of solution under the influence of the local geometry of $\G$ and the potential $h$.
\end{abstract}
\maketitle
\section{Introduction}\label{Intro}
We consider the following Hardy-Sobolev inequality with cylindrical weight: for $N\geq 3$, $0\leq k \leq N-1 \textrm{ and $ 0\leq \s\leq 2$}$, we have
\begin{equation}\label{Hardy-Sobolev11}
\int_{\R^N} |\nabla v|^2 dx \geq C \biggl(\int_{\R^N} |z|^{-\s} |v|^{2^*_\s} dx\biggl)^{2/2^*_\s}\qquad \textrm{ for all $v \in \calD^{1,2}(\R^N)$,}
\end{equation}
where $x=(t,z)\in \R^k \times \R^{N-k}$, $C$ is a positive constant depending only on  $N$, $k$ and $\s$,  $2^*_\s:= \frac{2(N-\s)}{N-2}$ is the critical Hardy-Sobolev exponent and $\calD^{1,2}(\R^N)$ is the completion of $C^\infty_c(\R^N)$ with respect to the norm 
$$
v \longmapsto \left(\int_{\R^N}|\n v|^2 dx\right)^{1/2}.
$$
For $\s=0$, inequality \eqref{Hardy-Sobolev11} corresponds to the following {\it classical Sobolev} inequality:
\begin{equation}\label{Sobolev11}
\int_{\R^N} |\nabla v|^2 dx \geq C \biggl(\int_{\R^N} |v|^{2^*_0} dx\biggl)^{2/2^*_0}\qquad \textrm{ for all $v \in \calD^{1,2}(\R^N)$}.
\end{equation}
In this case, the best constant (denoted $S_{N,0}$) is achieved by the function $w(x)=c\left(1+|x|^2\right)^{\frac{2-N}{2}}$ and hence the value  $S_{N,0}=N\left(N-2\right) \left[\Gamma\left(N/2\right)/ \Gamma\left(N\right)\right]$ explicitly (see Aubin \cite{Aubin}, Lieb \cite{LIEB} and Talenti \cite{rrtt}). Here $\G$ is the classical Euler function.\\

For $\s=2$ and $k\neq N-2$, inequality \eqref{Hardy-Sobolev11} corresponds to the following {\it classical Hardy} inequality:
\begin{equation}\label{Hardy11}
\int_{\R^N} |\nabla v|^2 dx \geq \left(\frac{N-k-2}{2}\right)^2\int_{\R^N} |z|^{-2} |v|^{2} dx \qquad \textrm{ for all $v \in \calD^{1,2}(\R^N)$}.
\end{equation}
The constant $\left(\frac{N-k-2}{2}\right)^2$ is optimal but it is never achieved. This fact suggests that it is possible to improve this inequality, see Brezis-Vasquez \cite{BV} and references therein. For improved Hardy inequality on compact Riemannian manifolds, see the paper of the author \cite{THIAM}.
For $\s\in(0,2)$, the best constant in \eqref{Hardy-Sobolev11} is given by 
\be\label{eq:Hard-Sob-sharp11}
S_{N,\s}:=\inf \left\lbrace \displaystyle\int_{\R^N} |\n u|^2 dx, \quad u\in \calD^{1,2}\left(\R^N\right)  \quad \textrm{and $\displaystyle\int_{\R^N} |z|^{-\s} |u|^{2^*_\s} dx=1$} \right\rbrace
\ee
and it is attained, see Badiale-Tarantello \cite{BT}. Moreover extremal functions  are cylindrical symmetric, see Fabbri-Mancini-Sandeep \cite{FMS}. However few of them are known explicitly. Indeed, when $k=0$, they are given up to scaling by 
$
w(x)=\left[\left(N-\s\right)\left(N-2\right)\right]^{\frac{N-2}{2(2-\s)}}\left(1+|x|^{2-\s}\right)^{\frac{2-N}{2-\s}}.
$
Thus the best constant is 
$$
S_{N,\s}:= (N-2)(N-\s) \biggl[\frac{w_{N-1}}{2-\s} \frac{\Gamma^2(N-\frac{\s}{2-\s})}{\Gamma(\frac{2(N-\s)}{2-\s})}\biggl]^{\frac{2-\s}{N-\s}},
$$
see Lieb \cite{LIEB}. When $\s=1$,  the authors in \cite{FMS} showed that the minimizers are given by
$$
w(x)=\left\lbrace(N-k)(k-1)\right\rbrace^{\frac{N-2}{2}}\left(\left(1+|z|\right)^2+|t|^2\right)^{\frac{2-N}{2}}
$$
up to scaling in the full variable and translations in the $t-$direction.\vspace{0.6cm}\\
In this paper, we consider a Hardy-Sobolev inequality in a bounded domain of the Euclidean space with singularity a closed submanifold of higher dimensional singularity. In particular, we let $\O$ be a bounded domain of $\R^N$, $N\geq 3$, and $h$ a continuous function defined on $\O$. Let  $\G \subset \O$ be a smooth closed submanifold in $\O$ of dimension $k$, with $1\leq k \leq N-2$. We are concerned with the existence of minimizers for the following Hardy-Sobolev best constant:
\be \label{eq:min-to-study}
\mu_{h, \s}(\O,\G):=\inf_{ u\in  H^1_0(\O) \setminus\{0\} }\frac{ \displaystyle \int_{\O} |\nabla u|^2 dx +\int_\O h u^2 dx }{  \displaystyle \left(  \int_{\O} \rho_\G^{-\s} |u|^{2^*_\s} dx \right)^{\frac{2}{2^*_\s}}},
\ee
where $\s\in(0,2)$,  $\displaystyle 2^*_\s:= \frac{2(N-\s)}{N-2}$ and $ \rho_\G(x):=\textrm{dist}(x,\G)$ is the distance function to $\G$.  Here and in the following, we assume that $-\D+h$ defines a  coercive bilinear form on $H^1_0(\O)$. We are interested with the effect of the local geometry of the submanifold $\G$ on  the existence of minimizer for $\mu_{h,\s}(\O,\G)$.\\

When $k=1$ (i.e. $\G$ is a curve), we have the following result due to the author and Fall \cite{Fall-Thiam}. 
\begin{theorem}\label{th:main1}
Let $N \geq 3$, $\s\in (0,2)$ and   $\O$  be a   bounded domain of $\R^N$. Consider  $\G$ a smooth closed curve contained in $\O$.
Let  $h$ be  a continuous function such that the linear operator $-\D+h$ is coercive. Then there exists a positive constant $C_{N,\s}$, only depending on $N$ and $\s$ with the property that if  there exists  $y_0\in \G$  such that 
\begin{equation}\label{eq:h-bound-main-th-1}
\begin{cases}
h(y_0)+C_{N,\s} |\k (y_0)|^2<0 &\quad \textrm{for $N\geq 4$}\\\\
m(y_0)>0 &\quad \textrm{for $N=3$}
\end{cases}
\end{equation}
then $\mu_{h,\s}\left(\O,\G\right) < S_{N,\s},$ and $\mu_{h,\s}\left(\O,\G\right)$ is achieved by a positive function. Here $\k:\G\to \R^N$ is the curvature vector of   $\G$ and $m: \O \to \R$ is the mass-the trace of the regular part of the  Green function of the operator $-\D+h$ with zero Dirichlet data.
\end{theorem}
This result shows the dichotomy between the case $N \geq 4$ and the case $N=3$ as in Brezis-Nirenberg \cite{BrezisNirenberg}, Druet \cite{Druet00}, Jaber \cite{Jaber1} et references therein.\\

Our main result deals with the case $2\leq k \leq N-2$ and $N\geq 4$. Then we have
\begin{theorem}\label{MainResult}
Let $N \geq 4$, $\s\in (0,2)$ and   $\O$  be a   bounded domain of $\R^N$. Consider  $\G$ a smooth closed submanifold  contained in $\O$ of dimension $k$ with $2\leq k \leq N-2$.
Let  $h$ be  a continuous function such that the linear operator $-\D+h$ is coercive. Then there exists  positive constants $C^1_{N,\s}$ and  $C^2_{N,\s}$, only depending on $N$ and $\s$ with the property that if  there exists  $y_0\in \G$  such that 
\begin{equation}\label{eq:h-bound-main-th-11}
C^1_{N,\s} H^2(y_0)+C^2_{N,\s} R_g(y_0)+h(y_0)<0
\end{equation}
then $\mu_{h,\s}\left(\O,\G\right) < S_{N,\s},$ and $\mu_{h,\s}\left(\O,\G\right)$ is achieved by a positive function. Here $H^2$ and $R_g$ are respectively the norms of the mean curvature and the scalar curvature of $\G$.
\end{theorem}
The explicit  values of $C^1_{N,\s}$ and $C^2_{N,\s}$ appearing  in \eqref{eq:h-bound-main-th-11}  are given by weighted integrals involving partial derivatives of $w$, a minimizer for $S_{N,\s}$, see Proposition \ref{Proposition2} below.
When $k=1$ then $R_g(x_0)=0$. Hence $H=\kappa$, so that we recover Theorem \ref{th:main1}.\\

In the litterature several authors studied Hardy-Sobolev inequalites in domains of the Euclidean space and in Riemannian manifolds, see \cite{CL, DN, GK, GR, GR1, GR2, GR3, GR4, GY, Jaber1, Jaber2, LL} and references therein. For instance, we let $\O$ to be a smooth bounded domain of $\R^N$ with $0\in \O$ and consider the following Hardy-Sobolev constant
\begin{equation}\label{MinimizationProblem}
\mu_\s\left(\O\right):=\inf \left\lbrace \displaystyle\int_{\O} |\n u|^2 dx, \quad u\in H^1_0\left(\O\right)  \quad \textrm{and $\displaystyle\int_\O |x|^{-\s} |u|^{2^*_\s} dx=1$} \right\rbrace,
\end{equation}
with $\s\in[0,2)$. It is well known that the value of $\mu_\s\left(\O\right)$ is independent of $\O$ thanks to scaling invariant. Moreover $\mu_\s\left(\O\right)=S_{N,\s}$ given by \eqref{eq:Hard-Sob-sharp11} and it is not attained for all bounded domains, see Ghoussoub-Yuan \cite{GY} and Struwe \cite{Struwe}. However the situation changes when we add a little perturbation. For example, let $h$ be a continuous function on $\O$. Consider the following Hardy-Sobolev best constant
\begin{equation}\label{BrezisNirenbergProblem}
\mu_{h, \s}\left(\O\right):=\inf \left\lbrace \displaystyle\int_{\O} |\n u|^2 dx+\int_\O h u^2 dx, \quad u\in H^1_0\left(\O\right)  \quad \textrm{and $\displaystyle\int_\O |x|^{-\s} |u|^{2^*_\s} dx=1$} \right\rbrace.
\end{equation} 
 When $\s=0$, \eqref{BrezisNirenbergProblem} corresponds to the famous Brezis-Nirenberg problem (see \cite{BrezisNirenberg}) and when $\s=2$, this kind of problem was study by the author on compact Riemannian manifolds, see \cite{THIAM}.
In the non-singular case $(\s=0)$, authors in \cite{BrezisNirenberg} showed that, for $N\geq 4$ it is enough that $h(y_0)<0$ to get minimizer for some $y_0\in \O$.  While for $N=3$, the problem is no more local and existence of minimizers is guaranted by the positiveness of a certain mass-the trace of the regular part of the Green function of the operator $-\D+h$ with zero Dirichlet data, see \cite{Druet00, Druet11}. Related references for this Brezis-Nirenberg type problem are Druet \cite{Druet000}, Hebey-Vaugon \cite{HV1, HV2}, Egnell \cite{Egnell} and references therein.\\
When $\s=2$ and $h\equiv \l$ is  a real parameter and $\O$ is replaced by a compact Riemannian manifold, then  the author in \cite{THIAM} proved the existence of a threshold $\l^*\left(\O\right)$ such that the best constant in \eqref{BrezisNirenbergProblem} has a solution if and only if $\l<\l^*$. See also \cite{Thiam3}.\vspace{0.6cm}\\
A very interesting case in the litterature is when $0\in \de \O$. The result of the attainability for the Hardy-Sobolev best constant $\mu_\s\left(\O\right)$ defined in \eqref{MinimizationProblem} is quite different from that in the situation where $0\in \O$.
The fact that things may be different when $0\in \de \O$ first emerged in the paper of Egnell \cite{Egnell} where he considers open cones of the form $C=\lbrace x\in \R^N; x=r\theta, \theta \in D \textrm{ and $r>0$}\rbrace$ where the base $D$ is a connected domain of the unit sphere $S^{N-1}$ of $\R^N$. Egnell showed that $\mu_\s\left(C\right)$ is then attained for $0<s<2$ even when $\bar{C}\neq \R^N$.  
Later Ghoussoub and Kang in \cite{GK} showed that if all the principal curvatures of $\de \O$ at $0$ are negative then $\mu_\s\left(\O\right)<\mu_\s\left(\R^{N}_+\right)$
and it is achieved. Demyanov and Nazarov in \cite{DN} proved that the extremals for $\mu_\s\left(\O\right)$
exist when $\O$ is average concave in a neighborhood of the origin. Later Ghoussoub and Robert in \cite{GR1} proved the existence of extremals when the boundary is smooth and the mean curvature at $0$ is negative. For more results in this direction and generalizations, we refer to Ghoussoub-Robert \cite{GR2, GR3, GR4}, Chern-Lin \cite{CL}, Lin-Li \cite{LL}, Lin-Walade \cite{LW1, LW2}, the Fall, Minlend and the author \cite{FMT} and references therein.\\

The proof of Theorem \ref{MainResult} rely on test function methods. Namely to build appropriate test functions allowing to compare $\mu_{h,\s}(\O,\G)$ and $ S_{N,\s}$.  While it  always holds that $\mu_{h,\s}(\O,\G)\leq S_{N,\s}$, our main task is  to find a function for which  $\mu_{h,\s}(\O,\G)<S_{N,\s}$. This then allows to recover compactness and thus every minimizing sequence  for $ \mu_{h,\s}(\O,\G)$ has a subsequence which converges to   a minimizer. Building these approximates solutions requires to have sharp decay estimates of a  minimizer $w$ for $S_{N,\s}$, see Lemma \ref{cor:dec-est-w} below.    In Section \ref{FatouElaj}, we prove existence result when $\mu_{h,\s}\left(\O, \G\right)<S_{N,\s}$. In Section \ref{Section1}, we give some preliminaries results. In Section \ref{SectionCompare},  we  build continuous familly  of test functions $(u_\e)_{\e>0}$ concentrating at a point $y_0\in \G$ which yields $\mu_{h,\s}(\O,\G)<S_{N,\s}  $, as $\e\to 0$, provided \eqref{eq:h-bound-main-th-11} holds.\\
\section{Preliminaries }\label{Section1}
Let $\G\subset \R^N$ be  a smooth closed submanifold of dimension $k$ with $2\leq k\leq N-2$. For $y_0 \in \G$, we let $(E_1;\dots; E_k)$ be an orthonormal basis of $T_{y_0}\G$, the tangent space of $\G$ at $y_0$. For $r>0$ small, a neighborhood of $y_0\in \G$ can be parametrized by the mapping $f :B_{\R^k} (0,r) \to \G$ defined  by
$$
f\left(t\right):=\textrm{exp}_{y_0}^\G\left(\sum_{a=1}^k t_a E_a\right),
$$
where $\textrm{exp}_{y_0}^\G$ is the exponential map of $\G$ at $y_0$ and $B_{\R^k} (0,r)$ is the ball of $\R^k$ centered at $0$ and of radius $r$.
We choose a smooth   orthonormal frame field $\left(E_{k+1}(f(t));...;E_N(f(t))\right)$ on the normal bundle of $\G$ such that $\left(E_1\left(f(t)\right);...;E_N(f(t))\right) $ is an oriented basis of $\R^N$ for every $t\in B^k_r$, with $E_i(f(0))=E_i$. We fix the following notation, that will be used  a lot in the paper,
$$
 Q_r:=B_{\R^k} (0,r) \times B_{\R^{N-k}}(0,r) ,
$$
where $B_{\R^m}(0,r)$ denotes the ball in $\R^m$ with radius $r$ centered at the origin.
 Provided $r>0$  small, the map $F_{y_0}: Q_r\to \O$, given by 
\begin{equation}\label{Parametrization}
 (t,z)\mapsto  F_{y_0}(t,z):= f(t)+\sum_{i=2}^N z_i E_i(f(t)),
\end{equation}
is smooth and parameterizes a neighborhood of $y_0=F_{y_0}(0,0)$. We consider $\rho_\G:\O\to \R$ the distance function to the submanifold given by 
$$
\rho_\G(t)=\min_{\ov t\in \G}|t-\ov t|.
$$
In the above coordinates, we have 
\begin{equation}\label{eq:rho_Gamm-is-mod-z}
\rho_\G\left(F_{y_0}(x)\right)=|z| \qquad\textrm{ for every $x=(t,z)\in Q_r.$} 
\end{equation}
Since the basis $\lbrace E_i \rbrace_i$ is orthonormal, then for every $t\in B_{\R^k} (0,r)$; $a, b=1, \cdots, k$ and $i,j=k+1,\dots N$,  there exists real numbers $\G^i_{ab}\left(f(t)\right)$ and  $\b^i_{ja}(f(t))$ such that we can write
\begin{equation}\label{Def-Curvature}
dE_i \circ \frac{\de f}{\de t_a}=-\sum_{b=1}^k \Gamma^i_{ab} \frac{\de f}{\de t_b}+ \sum_{\stackrel{j=k+1}{i\neq j}} \b^i_{ja} E_j.
\end{equation}
The quantity $\G^i_{ab}\left(f(t)\right)$ and $\b^i_{ja}(f(t))$ are the second fundamental form and 
the "torsion" of $\G$. The norms of the second fundamental form and the mean curvature are then given respectively by
$$
\G:=\left(\sum_{ab=1}^k \sum_{i=k+1}^N \left(\Gamma^i_{ab}\right)^2\right)^{1/2} \qquad \textrm{and} \qquad H:=\left(\sum_{i=k+1}^N \left(\sum_{a=1}^k\Gamma^i_{aa}\right)^2\right)^{1/2}.
$$
We note that provided $r>0$ small, $\Gamma^i_{ab}$ and $\b^i_{ja}$ are smooth functions. Moreover, it is easy to see that 
\be\label{eq:tau-antisymm}
\b^i_{ja}(f(t))=-\b^j_{ia}(f(t)) \qquad\textrm{ for $i,j=2,\dots, N$ and $a=1, \cdots, k$ }.
\ee
Next, we derive the expansion of the metric induced by the parameterization $F_{y_0}$ defined above.
For $x=(t,z) \in Q_r$, we define 
$$
g_{ab}(x):=   {\de_{t_a} F_{y_0}}(x) \cdot   {\de_{t_b} F_{y_0}} (x) , \qquad g_{ai}(x):=   {\de_{t_a} F_{y_0}} (x)\cdot   {\de_{z_i} F_{y_0}}(x)
$$
and 
$$
g_{ij}(x):=   {\de_{z_j} F_{y_0}}(x) \cdot   {\de_{z_i} F_{y_0}}(x).
$$
Then we have the following
\begin{lemma}\label{MaMetric}
For any $a,b=1,\cdots, k$ and for any $i,j=k+1, \cdots,N$, we have
{\small
\begin{align*}
\displaystyle g_{ab}(x) &=\d_{ab}-2\sum_{i=k+1}^N z_i \Gamma^i_{ab}+\sum_{ij=k+1}^N \sum_{c=1}^k z_i z_j \Gamma^i_{ac} \Gamma^j_{bc}\\
&+\sum_{ij=k+1}^N \sum_{\stackrel{l=k+1}{\stackrel{l\neq i}{l\neq j}}}^N z_i z_j \b^i_{la} \b^j_{lb} -\frac{1}{3} \sum_{cd=1}^k R_{acbd}(x_0) t_c t_d+O\left(|x|^3\right)\\
 g_{ai}(x)&= \sum_{\stackrel{j=k+1}{j\neq i}}^N z_j \b^j_{ia}\qquad \textrm{and}\qquad 
 g_{ij}(x)=\d_{ij},
\end{align*}
}
where the curvature terms $\Gamma^i_{ab}$ and $\b_{ia}^j$ are computed at the point $f(t)$.
\end{lemma}
\begin{proof}
We use the expression in \eqref{Parametrization} to get
\begin{equation}\label{Derives}
\frac{\de F}{\de t_a}= \frac{\de f}{\de t_a}+\sum_{i=k+1}^N z_i dE_i \circ \frac{\de f}{\de t_a} \qquad \textrm{and} \qquad \frac{\de F}{\de z_i}= E_i.
\end{equation}
Then using \eqref{Derives} and the fact that $\displaystyle\frac{\de f}{\de t_a} \in T_{f(t)} \G,$  we easily get
$$
g_{ai}(x)= \sum_{\stackrel{j=k+1}{j\neq i}}^N z_j \b^j_{ia} \qquad \textrm{and} \qquad g_{ij}(x)=\d_{ij}.
$$
We have also that
\begin{equation}\label{Hum3}
\begin{array}{ll}
\displaystyle g_{ab}(x)=\langle \frac{\de f}{\de t_a}, \frac{\de f}{\de t_b}\rangle&\displaystyle+\sum_{i=k+1}^N z_i \langle dE_i\circ \frac{\de f}{\de t_a}, \frac{\de f}{\de t_b}\rangle\\\
&\displaystyle+\sum_{i=k+1}^N z_i \langle dE_i\circ \frac{\de f}{\de t_b}, \frac{\de f}{\de t_a}\rangle+\sum_{ij=k+1}^N z_i z_j \langle dE_i\circ \frac{\de f}{\de t_a}, dE_j\circ \frac{\de f}{\de t_b}\rangle.
\end{array}
\end{equation}
The expansion of the induced  metric 
$
\tilde{g}_{ab}=\langle \frac{\de f}{\de t_a}, \frac{\de f}{\de t_b}\rangle
$
in the local chart of the exponential map is given by
\begin{equation}\label{Hum4}
\displaystyle\tilde{g}_{ab}(x)=\d_{ab}-\frac{1}{3} \sum_{cd=1}^k R_{acbd}(x_0) t_c t_d +O\left(|t|^3\right),
\end{equation}
where the $R_{acbd}$ are the components of the tensor curvature of $\G$, see \cite{Gray}.
We then plug \eqref{Hum4} in \eqref{Hum3} to get
$$
\begin{array}{ll}
\displaystyle g_{ab}(x)=&\displaystyle\d_{ab}-2\sum_{i=k+1}^N z_i \Gamma^i_{ab}+\sum_{ij=k+1}^N \sum_{c=1}^k z_i z_j \Gamma^i_{ac} \Gamma^j_{bc}\\\\
&\displaystyle+\sum_{ij=k+1}^N \sum_{\stackrel{l=k+1}{\stackrel{l\neq i}{l\neq j}}}^N z_i z_j \b^i_{la} \b^j_{lb} -\frac{1}{3} \sum_{cd=1}^k R_{acbd}(x_0) t_c t_d+O\left(|x|^3\right).
\end{array}
$$
This ends the proof.
\end{proof}
We will need the following result deduced from Lemma \ref{MaMetric}.
\begin{lemma}\label{DetInv}
In a small neighborhood of the point $y_0\in \O$ the expansion of the square root of the determinant of the metric is given by
\begin{equation}\label{DeterminantMetric}
\begin{array}{ll}
\displaystyle\sqrt{|g|}(x)=1-\sum_{i=k+1}^N z_i H^i&\displaystyle-\frac{1}{2}\sum_{ij=k+1}^N \sum_{ab=1}^k z_i z_j \Gamma^i_{ab} \Gamma^j_{ab}+\sum_{ij=k+1}^N z_i z_j H^i H^j\\\\
&\displaystyle-\frac{1}{6} \sum_{cd=1}^k Ric_{cd}(x_0) t_c t_d+O\left(|x|^3\right).
\end{array}
\end{equation}
Moreover the components of the inverse of the metric are 
\begin{equation}\label{InverseMetric}
\begin{array}{ll}
g^{ab}(x)&\displaystyle=\d_{ab}+2\sum_{i=k+1}^N z_i \Gamma^i_{ab}+3 \sum_{ij=k+1}^N \sum_{c=1}^k z_i z_j \Gamma^i_{ac} \Gamma^j_{bc}+\frac{1}{3} \sum_{cd=1}^k R_{acbd}(x_0) t_c t_d+O\left(|x|^3\right)\\\\
g^{ai}(x)&\displaystyle=-\sum_{j=k+1}^N z_j \b^j_{ia}-2\sum_{c=1}^k \sum_{lm=k+1}^N z_l z_m \Gamma^l_{ac} \Gamma^m_{ac}
+O\left(|x|^3\right)\\\\
g^{ij}(x)&\displaystyle=\d_{ij}+\sum_{c=1}^k \sum_{lm=k+1}^N z_l z_m \b^l_{ic} \b^m_{jc}+O\left(|x|^3\right).
\end{array}
\end{equation}
\end{lemma}
\begin{proof}
We can write $g(x)=I+A$. Then we have the classical expansion
\begin{equation}\label{DetFormula}
\sqrt{\det\left(I+A\right)}(x)=1+\frac{\tr A}{2}+\frac{\left(\tr A\right)^2}{4}-\frac{\tr \left(A^2\right)}{4}+O\left(|A|^3\right),
\end{equation}
where we have
\begin{equation}\label{H1}
\frac{tr A}{2} =
-\sum_{i=k+1}^N z_i H^i+\frac{1}{2}\sum_{ij=k+1}^N \sum_{ab=1}^k z_i z_j \Gamma^i_{ab} \Gamma^j_{ab}+\frac{1}{2}\sum_{\stackrel{ijl=k+1}{\stackrel{l\neq i}{l\neq j}}}^N z_i z_j \b^i_{la} \b^j_{la}
-\frac{1}{6} \sum_{cd=1}^k Ric_{cd}(x_0) t_c t_d+O\left(|x|^3
\right)
\end{equation}
and 
\begin{equation}\label{H2}
\frac{\left(tr A\right)^2}{4}=\sum_{ij=k+1}^N z_i z_j H^i H^j+O\left(|x|^3\right),
\end{equation}
where for $i=k+1, \cdots, N$ we have
$$
H^i=\sum_{a=1}^k \Gamma^i_{aa}
$$
 are the components of the mean curvature of $\G$. Moreover using the fact that the matrix $A$ is symmetric, we get
$$
tr A^2=\sum_{\a=1}^N \left(A^2\right)_{\a\a}=\sum_{\a=1}^N \left(\sum_{\b=1}^N A_{\a\b}(x) A_{\a\b}(x)\right).
$$
Then
\begin{align*}
-\frac{tr A^2}{4}&\displaystyle=-\frac{1}{4} \left(\sum_{ab=1}^k A_{ab}^2(x)+2\sum_{a=1}^k \sum_{i=k+1}^N A_{ai}^2(x)+
\sum_{ij=k+1}^N A_{ij}^2(x)
\right).
\end{align*}
Therefore
\begin{equation}\label{H3}
-\frac{tr A^2}{4}=-\sum_{ij=k+1}^N \sum_{ab=1}^k z_i z_j \Gamma^i_{ab} \Gamma^j_{ab}-\frac{1}{2} \sum_{a=1}^k \sum_{\stackrel{ijl=k+1}{\stackrel{i\neq l}{j\neq l}}}^N z_i z_j \b^i_{la} \b^j_{la}.
\end{equation}
By \eqref{DetFormula}, \eqref{H1}, \eqref{H2} and \eqref{H3}, we finally obtain
\begin{align*}
\sqrt{|g|}(x)=1&-\sum_{i=k+1}^N z_i H^i-\frac{1}{2}\sum_{ij=k+1}^N \sum_{ab=1}^k z_i z_j \Gamma^i_{ab} \Gamma^j_{ab}\\\
&+\sum_{ij=k+1}^N z_i z_j H^i H^j
-\frac{1}{6} \sum_{cd=1}^k Ric_{cd}(x_0) t_c t_d+O\left(|x|^3\right).
\end{align*}
We write
$$
g(x)=I+A(x)+B(x)+O\left(|x|^3\right),
$$
where $A$ and $B$ are symmetric matrix given by
$$
A_{ab}(x)=-2 \sum_{i=k+1}^N z_i \Gamma^i_{ab}; \qquad A_{ai}(x)= \sum_{j=k+1}^N z_j \b^j_{ia} \qquad \textrm{and} \qquad A_{ij}(x)=0
$$
and
$$
B_{ab}(x)=\sum_{ij=k+1}^N \sum_{c=1}^k z_i z_j \Gamma^i_{ac} \Gamma^j_{bc}+\sum_{ij=k+1}^N \sum_{\stackrel{l=k+1}{\stackrel{l\neq i}{l\neq j}}}^N z_i z_j \b^i_{la} \b^j_{lb} -\frac{1}{3} \sum_{cd=1}^k R_{acbd}(x_0) t_c t_d
$$
and
$$
B_{ai}(x)=B_{ij}(x)=0.
$$
It's clear that the inverse of the metric $g^{-1}$ is given by
$$
g^{-1}(x)=I-A(x)-B(x)+A^2(x)+O\left(|x|^3\right).
$$
This yields
\begin{align*}
g^{ab}(x)&\displaystyle=\d_{ab}-A_{ab}(x)-B_{ab}(x)+\sum_{c=1}^k A_{ac}(x) A_{bc}(x)+\sum_{i=k+1}^N A_{ai}(x) A_{bi}(x)+O\left(|x|^3\right)\\\\
g^{ai}(x)&\displaystyle=-A_{ai}(x)-B_{ai}(x)+\sum_{c=1}^k A_{ac}(x) A_{ic}(x)+\sum_{j=k+1}^N A_{aj}(x) A_{ij}(x)+O\left(|x|^3\right)\\\\
g^{ij}(x) &\displaystyle=\d_{ij}-A_{ij}(x)-B_{ij}(x)+\sum_{c=1}^k A_{ic}(x) A_{jc}(x)+\sum_{l=k+1}^N A_{il}(x) A_{jl}(x)+O\left(|x|^3\right).
\end{align*}
Hence we obtain that
\begin{equation}\label{HV}
\begin{array}{ll}
g^{ab}(x)&\displaystyle=\d_{ab}+2\sum_{i=k+1}^N z_i \Gamma^i_{ab}+3 \sum_{ij=k+1}^N \sum_{c=1}^k z_i z_j \Gamma^i_{ac} \Gamma^j_{bc}+\frac{1}{3} \sum_{cd=1}^k R_{acbd}(x_0) t_c t_d+O\left(|x|^3\right)\\\\
g^{ai}(x)&\displaystyle=-\sum_{j=k+1}^N z_j \b^j_{ia}-2\sum_{c=1}^k \sum_{lm=k+1}^N z_l z_m \Gamma^l_{ac} \Gamma^m_{ac}+O\left(|x|^3\right)\\\\
g^{ij}(x)&\displaystyle=\d_{ij}+\sum_{c=1}^k \sum_{lm=k+1}^N z_l z_m \b^l_{ic} \b^m_{jc}+O\left(|x|^3\right).
\end{array}
\end{equation}
This ends the proof of the lemma.
\end{proof}
We consider  the  best constant for the cylindrical Hardy-Sobolev inequality   
$$
S_{N,\s}= \min \left\{ \int_{\R^N} |\n w|^2 dx\,:\,w\in \cD^{1,2}(\R^N),\,    \int_{\R^N} |z|^{-\s} |w|^{2^*_\s} dx=1 \right\}.
$$
As mentioned in the first section,  it is   attained  by a positive function $w\in \cD^{1,2}(\R^N)$, satisfying
\begin{equation}\label{ExpoEA}
-\D w=S_{N,\s} |z|^{-\s} w^{2^*_\s-1} \qquad \textrm{ in } \R^N,
\end{equation}
see e.g. \cite{BT}. Moreover from  \cite{FMS}, we have   
\begin{equation}\label{eq:AE}
w(x)=w(t,z)=\theta\left(|t|, |z|\right) \qquad \textrm{ for a function} \qquad \theta:\R_+ \times \R_+ \to \R_+.
\end{equation}
We will need the following preliminary result in the sequel.
\begin{lemma}\label{cor:dec-est-w}
Let $w$ be a ground state for $S_{N,\s}$ then there exist positive   constants $C_1,C_2$, only depending on $N$ and $\s$, such that 
\item[(i)] For every $x\in \R^N$  
  \begin{equation}\label{DecayEstimates111}
\frac{C_1}{1+|x|^{N-2}}\leq w(x) \leq \frac{C_2}{1+|x|^{N-2}} .
\end{equation}

\item[(ii)] For   $|x|= |(t,z)|\leq 1$  
$$ 
|\n w (x)|+ |x| |D^2 w (x)|\leq C_2 |z|^{1-\s}
$$
\item[(iii)] For   $|x|= |(t,z)|\geq 1$ 
$$
|\n w(x)|+ |x| |D^2 w(x)|\leq C_2 \max(1, |z|^{-\s})|x|^{1 -N}.
$$
\end{lemma}
Fabbri, Mancini and Sandeep proved $(i)$ in \cite{FMS}. The proof of $(ii)$ and $(iii)$ are done by the Fall and the author in \cite{Fall-Thiam}.
%
%
%
\section{Existence Result}\label{FatouElaj}
 Let $\O$ be a bounded domain of $\R^N$, $N\geq 3$, and $h$ a continuous function on $\O$. Let  $\G$ be a smooth closed submanifold contained in $\O$.  We consider 
\be \label{eq:mu-h-O-G-appendix}
\mu_{h,\s}(\O,\G):=\inf_{ u\in  H^1_0(\O)  }\frac{ \displaystyle \int_{\O} |\nabla u|^2 dx +\int_\O h u^2 dy }{  \displaystyle \left(  \int_{\O} \rho_\G^{-\s} |u|^{2^*_\s} dy \right)^{\frac{2}{2^*_\s}}}.
\ee
We also recall  that
\be\label{eq:S-N-sig-appendix}
S_{N,\s}=\inf_{  v\in \cD^{1,2}(\R^N)  }\frac{ \displaystyle  \int_{\R^N} |\nabla v|^2 dx}{   \displaystyle \left(  \int_{\R^N} |z|^{-\s} |v|^{2^*_\s} dx \right)^{\frac{2}{2^*_\s}} },
%
\ee
with $x=(t,z)\in \R\times \R^{N-1}$.
  Our aim in this section is to show   that   { if } $\mu_{h,\s}\left(\O,\G \right)< S_{N,\s}$   then the best constant  $\mu_{h,\s}\left(\O,\G\right)$    is achieved. The argument of proof is standard. However, for sake of completeness, we add the proof. We start with the following
\begin{lemma}\label{OuiOui}
Let $\O$ be an open subset of $\R^N$, with $N\geq 3$, and let $\G\subset \O$ be a smooth closed submanifold.
Then for every $r>0$, there exist  positive constants $  c_r>0$, only  depending on $\O,\G,N,\s$ and $r$, such that  for every  $u \in H^1_0(\O)$  
$$
S_{N,\s} \left(\int_{\O} \rho^{-\s}_\G |u|^{2^*_\s} dy\right)^{2/2^*_\s} \leq(1+ r) \int_{\O} |\n u|^2dy+c_r \left[\int_\O u^2 dy+\left(\int_\O |u|^{2^*_\s} dy\right)^{2/2^*_\s}\right],
$$
where $2^*_\s=\frac{2(N-\s)}{N-2}$ and $\s\in (0, 2)$.
\end{lemma}
\begin{proof}
We let $r>0$ small. We can cover a tubular neighborhood of $\G$ by a finite number of  sets $\left(T^{y_i}_r\right)_{1\leq i \leq m}$  given by 
$$
T^{y_i}_r:=F_{y_i}\left(Q_r\right), \qquad \textrm{ with $y_i\in\G$. }
$$
We refer to  Section \ref{Section1} for the parameterization  $F_{y_i}:Q_r\to \O$. 
See e.g. \cite[Section 2.27]{Aubin}, there exists  $\left(\vp_i\right)_{1\leq i \leq m}$  a   partition of unity subordinated to this covering such that
\begin{equation}\label{eq:Part-unity}
\sum_i^{m } \vp_i  =1 \qquad \textrm{and} \qquad |\n \vp_i ^{\frac{1}{2^*_\s}} |\leq K \qquad \textrm{ in } U:=\displaystyle \cup_{i=1}^{m} T^{y_i}_r,
\end{equation}
for some constant $K>0$.
We define
\be \label{eq:def-psi_i}
\psi_i(y):=\vp_i^{\frac{1}{2^*_\s}} (y) u(y) \qquad \textrm{ and } \qquad \ti{\psi_i}(x)=\psi_i (F_{y_i}(x)).
\ee
Recall that $\rho_\G\geq C>0$ on  $\O \setminus U$, for some positive constant $C>0$. Therefore,  since $\frac{2}{2^*_\s}<1$, by \eqref{eq:def-psi_i} we get
\begin{align} \label{eq:Sum-append}
\displaystyle\left(\int_\O \rho^{-\s}_\G |u|^{2^*_\s} dy+\right)^{2/2^*_\s} 
&\displaystyle\leq   \left(\int_U \rho^{-\s}_\G \left| u \right|^{2^*_\s} dy\right)^{2/2^*_\s}+ \left(\int_{\O\setminus U} \rho^{-\s}_\G |u|^{2^*_\s} dy\right)^{2/2^*_\s}\nonumber\\
&\displaystyle \leq \left( \sum_i^{m} \int_{T^{y_i}_r} \rho^{-\s}_\G  |\psi_i|^{2^*_\s} dy\right)^{2/2^*_\s}+c_r  \left(\int_\O |u|^{2^*_\s} dy\right)^{2/2^*_\s} \nonumber\\
&\displaystyle \leq\sum_i^{m} \left(  \int_{T^{y_i}_r} \rho^{-\s}_\G  |\psi_i|^{2^*_\s} dy\right)^{2/2^*_\s}+c_r  \left(\int_\O |u|^{2^*_\s} dy\right)^{2/2^*_\s}.
\end{align}
By  change of variables and Lemma \ref{DetInv}, we have
\begin{align*}
\left(\int_{T^{y_i}_r} \rho^{-\s}_\G |\psi_i|^{2^*_\s} dy\right)^{2/2^*_\s} &=\left(\int_{Q_r} |z|^{-\s} |\ti{\psi}_i|^{2^*_\s} \sqrt{|g|}(x) dx\right)^{2/2^*_\s}\\
&\leq \left(1+c r\right) \left(\int_{Q_r} |z|^{-\s} |\ti{\psi}_i|^{2^*_\s}  dx\right)^{2/2^*_\s}.
\end{align*}
In addition by the Hardy-Sobolev best constant \eqref{eq:S-N-sig-appendix}, we have
$$
S_{N,\s}  \left(\int_{Q_r} |z|^{-\s} |\ti{\psi}_i|^{2^*_\s}  dx\right)^{2/2^*_\s}\leq  \left( \int_{Q_r} |\n \ti{\psi}_i|^2 dx \right)^{2/2}.
$$
Therefore  by change of variables and Lemma \ref{DetInv},  we get
\begin{align*}
S_{N,\s}&\left(\int_{T^{y_i}_r} \rho^{-\s}_\G |\psi_i|^{2^*_\s} dy\right)^{2/2^*_\s}  \leq \left(1+cr\right)\int_{Q_r} |\n \ti{\psi}_i|^2 dx\\
&\leq \left(1+ c' r\right) \int_{T^{y_i}_r} |\n (\phi_i^{\frac{1}{2^*_\s}}  u)|^2 dy=  \left(1+ c'r\right) \int_{T^{y_i}_r}|\phi_i^{\frac{1}{2^*_\s}}\n u+ u \n \phi_i^{\frac{1}{2^*_\s}}|^2 dy+ c_r  \int_\O | u|^2 dy.
\end{align*}
Applying Young's inequality   using \eqref{eq:Part-unity} and \eqref{eq:def-psi_i}, we find that
\begin{align*}
S_{N,\s}\left(\int_{T^{y_i}_r} \rho^{-\s}_\G |\psi_i|^{2^*_\s} dy\right)^{2/2^*_\s}& \leq   \left(1+ c' r\right)(1+\e) \int_{T^{y_i}_r} \phi_i^{\frac{2}{2^*_\s}}|\n u|^2 dy+ c_r(\e)  \int_\O | u|^2 dy\\
&\leq   \left(1+ c' r\right)(1+\e) \int_{T^{y_i}_r} |\n u|^2 dy+ c_r(\e)  \int_\O | u|^2 dy.
\end{align*}
Summing for $i$ equal $1$ to $m$,  we get 
\begin{align*}
S_{N,\s}  \sum_{i=1}^{m}\left(\int_{T^{y_i}_r} \rho^{-\s}_\G |\psi_i|^{2^*_\s} dy\right)^{2/2^*_\s}\leq     \left(1+ c' r\right)(1+\e)    \left( \int_{\O} |\n u|^2 dy\right)^{1/2}+ c_r(\e)  \left(\int_\O | u|^2 dy\right)^{1/2}.
\end{align*}
This together with \eqref{eq:Sum-append} give
$$
S_{N,\s} \left(\int_\O \rho^{-\s}_\G |u|^{2^*_\s} dy\right)^{2/2^*_\s}\leq  \left(1+ c' r\right)(1+\e)   \int_{\O} |\n u|^2 dy + c_r(\e)  \int_\O | u|^2 dy +c_r \left(\ \int_\O |u|^{2^*_\s} dy\right)^{2/2^*_\s}.
$$
Since $\e$ and $r$ can be  chosen arbitrarily small, we get the desired result.
\end{proof}
We can now prove the following existence result.
\begin{proposition}\label{Pro3333}
Consider $\mu_{h,\s}(\O,\G)$ and $S_{N,\s}$ given by \eqref{eq:mu-h-O-G-appendix} and \eqref{eq:S-N-sig-appendix} respectively.
Suppose that 
\begin{equation}\label{Assumption}
\mu_{h,\s}\left(\O,\G\right)<S_{N,\s}.
\end{equation}
Then $\mu_{h,\s}\left(\O,\G\right) $ is achieved by a positive function.
\end{proposition}
\begin{proof}
Let $\left(u_n\right)_{n\in \N}$ be a minimizing sequence for $\mu_{h,\s}\left(\O,\G\right)$ normalized so that
\begin{equation}\label{Minimization}
\int_\O \rho^{-\s}_\G |u|^{2^*_\s} dx=1 \quad \textrm{ and } \quad \mu_{h,\s}\left(\O,\G\right)=\int_\O |\n u_n|^2 dx +\int_\O h u_n^2 dx+o(1).
\end{equation}
By coercivity of $-\D+ h$, the sequence $\left(u_n\right)_{n\in \N}$ is bounded in $H^1_0(\O)$ and thus , up to a subsequence, 
$$
u_n \rightharpoonup u \qquad\textrm{weakly in } H^1_0(\O),
$$
and
\be \label{eq:-un-to-u-strong}
u_n \to  u \qquad \textrm{ strongly in } L^p(\O)\quad \textrm{for}\quad 1\leq p < 2^*_0:=\frac{2N}{N-2}. 
\ee
The weak convergence  in  $H^1_0(\O)$ implies that
\begin{equation}\label{SB1}
\int_\O |\n u_n|^2 dx= \int_\O |\n(u_n-u)|^2 dx+\int_\O |\n u|^2 dx+o(1).
\end{equation}
By Brezis-Lieb lemma \cite{BL} and the strong convergence in the Lebesgue spaces $L^p(\O)$, we have
\begin{equation}\label{SB2}
1=\int_\O \rho^{-\s}_\G |u_n|^{2^*_\s} dx=\int_\O \rho^{-\s}_\G|u-u_n|^{2^*_\s} dx+ \int_\O \rho^{-\s}_\G |u|^{2^*_\s} dx+o(1).
\end{equation}
By Lemma \ref{OuiOui}, \eqref{eq:-un-to-u-strong} ---note that $2^*_\s< 2^*_0$ , we then deduce that
\begin{equation}\label{SB3}
S_{N,\s} \left(\int_\O \rho^{-\s}_\G |u-u_n|^{2^*_\s} dx\right)^{2/2^*_\s} \leq(1+r) \int_\O|\n (u-u_n)|^2 dx+o(1).
\end{equation}
Using \eqref{SB1}, \eqref{SB2} and \eqref{SB3}, we have
\begin{align}\label{eq:near-exist}
\displaystyle S_{N,\s}\left(1-\int_\O \rho^{-\s}_\G |u|^{2^*_\s} dx\right)^{2/2^*_\s} &\displaystyle\leq \left(1+r\right) \left(\int_\O |\n u_n|^2 dx -\int_\O |\n u|^2 dx\right)+o(1) \nonumber\\
&\displaystyle= \left(1+r\right) \left(\mu_{h,\s}\left(\O,\G\right)-\int_\O h u_n^2 dx -\int_\O |\n u|^2 dx\right)+o(1) \nonumber\\
&\displaystyle= \left(1+r\right) \left(\mu_{h,\s}\left(\O,\G\right)-\int_\O h u^2 dx -\int_\O |\n u|^2 dx\right)+o(1) \nonumber\\
&\displaystyle\leq \left(1+r\right) \mu_{h,\s}\left(\O,\G\right)\left(1-\left(\int_\O \rho^{-\s}_\G |u|^{2^*_\s} dx\right)^{2/2^*_\s}\right)+o(1).
\end{align}
By  the concavity of the map $t\mapsto t^{2/2^*_\s}$ on $[0,1]$,  we have 
$$
 1 \leq \left(1-\int_\O \rho^{-\s}_\G |u|^{2^*_\s} dx\right)^{2/2^*_\s}+ \left(\int_\O \rho^{-\s}_\G |u|^{2^*_\s} dx\right)^{2/2^*_\s}.
$$
From this,  
then taking the limits respectively as $n \to +\infty$ and as $r \to 0$ in \eqref{eq:near-exist}, we find that
$$
\left[S_{N,\s}-\mu_{h,\s}(\O,\G) \right]\left(1-\left(\int_\O \rho^{-\s}_\G |u|^{2^*_\s} dx\right)^{2/2^*_\s}\right)\leq 0.
$$
Thanks to  \eqref{Assumption},  we then  get 
$$
1 \leq \int_\O \rho^{-\s}_\G |u|^{2^*_\s} dx .
$$
Since by \eqref{Minimization}  and Fatou's lemma,
$$
1=\int_\O \rho^{-\s}_\G |u_n|^{2^*_\s} dx \geq \int_{\O} \rho^{-\s}_\G |u|^{2^*_\s} dx,
$$
we conclude that
$$
\int_\O \rho^{-\s}_\G |u|^{2^*_\s} dx=1.
$$
It then follows from \eqref{Minimization} that $u_n\to u$ in $L^{2^*_\s}(\O;\rho_\G^{-\s})$ and  thus $u_n\to u$ in $H^1_0(\O)$.
Therefore $u$ is a minimizer for $\mu_{h,\s}(\O,\G) $.  Since $|u|$  is also a minimizer for $\mu_{h,\s}(\O,\G)$, we may assume that $u \gneqq0$. Therefore $u>0$ by the maximum principle.
\end{proof}

\section{Comparing $S_{N,\s}$ and $\mu_{h, \s}\left(\O, \G\right)$}\label{SectionCompare}
\begin{lemma}\label{lem:to-prove}
Let $v\in \cD^{1,2}(\R^N)$, $N\geq 3,$ satisfy $v(t,z)=\Theta(|t|,|z|)$, for some some function $\Theta: \R_+\times \R_+\to \R$. Then  for $0<r<R$, we have 
\begin{align*}
 \int_{\B_{R}\setminus \B_{r}}|\n v|^2_{g}\sqrt{|g|} dx= &\int_{\B_{R}\setminus \B_{r}}|\n v|^2  dx+\frac{3\G^2-2H^2}{k(N-k)}\int_{\B_R\setminus \B_r} |z|^2 |\n_tv|^2 dx\nonumber\\\\
 &+\frac{R_g(x_0)}{3k^2} \int_{\B_R\setminus \B_r} |t|^2 |\n_t v|^2 dx+\frac{H^2-\G^2/2}{N-k}\int_{\B_R\setminus \B_r} |z|^2 |\n v|^2 dx\\\\
 &-\frac{R_g(x_0)}{6k} \int_{\B_R\setminus \B_r} |t|^2 |\n v|^2 dx+O\left(\int_{\B_R\setminus \B_r} |x|^3 |\n v|^2 dx\right).
\end{align*}
\end{lemma}
\begin{proof}
It is easy to see that
\begin{align}\label{Exp00}
\int_{\B_{R}\setminus \B_{r}}|\n v|^2_{g}\sqrt{|g|} dx= \int_{\B_{R}\setminus \B_{r}}|\n v|^2  dx&+  \int_{\B_{R}\setminus \B_{r}} (|\n v|^2_{g} -|\n v|^2)\sqrt{|g|} dx\\\\
&+ \int_{\B_{R}\setminus \B_{r}}|\n v|^2 (\sqrt{|g|} -1)dx.
\end{align}
We recall that
\begin{align*}
|\n v|^2_{g}(x) -|\n v|^2(x)=  \sum_{\a\b=1}^N\left[ g^{\a\b}( x)- \d_{\a\b} \right] \de_{z_\a} v(x) \de_{z_\b} v(x).
\end{align*}
It then follows that 
\begin{align}\label{Exp0}
 \int_{\B_{R}\setminus \B_{r}}  \left[ |\n v|^2_{g} -|\n v|^2 \right]\sqrt{|g|}dx =& \sum_{ij=k+1}^N \int_{\B_{R}\setminus \B_{r}} \left[g^{ij}-\d_{ij}\right] \de_{z_i} v   \de_{z_j} v   \sqrt{|g|} dx  \nonumber\\
 &+2\sum_{a=1}^k\sum_{i=2}^N\int_{\B_{R}\setminus \B_{r}} g^{ia}\left(\de_{t_a} v   \de_{z_i} v \right) \sqrt{|g|}  dx  \\
 &+ \sum_{ab=1}^k\int_{\B_{R}\setminus \B_{r}} [g^{ab} -\d_{ab}]\left(\de_{t_a} v \de_{t_b} v  \right)\sqrt{|g|} dx  \nonumber.
\end{align}
We first use  Lemma \ref{MaMetric}, Lemma \ref{DetInv} and \eqref{eq:tau-antisymm}, to get
\begin{align}\label{Exp1}
 \displaystyle \sum_{ij=k+1}^N \int_{\B_{R} \setminus \B_{r}   } &\left[g^{ij} -\d_{ij}\right] \de_{z_i} v  \de_{z_j} v   \sqrt{|g|} \,dx\nonumber\\
& =\sum_{ij=k+1}^N \int_{\B_{R} \setminus \B_{r}} \biggl(\sum_{c=1}^k \sum_{lm=k+1}^N \b^l_{ic} \b_{jc}^m z_l z_m+O\left(|x|^3\right)\biggl) \frac{z_i z_j}{|z|^2} |\n_z v|^2 dx
\nonumber\\
&=\sum_{c=1}^k \sum_{ij=k+1}^N \int_{\B_{R} \setminus \B_{r}   } \b^j_{ic} \b^j_{ic} \frac{z_i^2 z_j^2}{|z|^2} |\n_z v|^2 dx\nonumber\\
&+\sum_{c=1}^k \sum_{{\stackrel{ij=k+1}{i\neq j}}}^N \int_{\B_{R} \setminus \B_{r}   } \b^i_{ic} \b^j_{jc} \frac{z_i^2 z_j^2}{|z|^2} |\n_z v|^2 dx\nonumber\\
&+\sum_{c=1}^k \sum_{{\stackrel{ij=k+1}{i\neq j}}}^N \int_{\B_{R} \setminus \B_{r}   } \b^j_{ic} \b^i_{jc} \frac{z_i^2 z_j^2}{|z|^2} |\n_z v|^2 dx
 +O\left( \int_{\B_{R }\setminus \B_{r}} |x|^3 |\n_z v|^2 dx\right).
 \nonumber\\
&=O\left( \int_{\B_{R }\setminus \B_{r}} |x|^3 |\n_z v|^2 dx\right).
\end{align}
Using again Lemma \ref{MaMetric} and Lemma \ref{DetInv}, it easy follows that
\begin{align}\label{Exp2}
\sum_{a=1}^k\sum_{i=2}^N\int_{\B_{R}\setminus \B_{r}} g^{ia}\left(\de_{t_a} v   \de_{z_i} v \right) \sqrt{|g|}  dx&=\sum_{a=1}^k\sum_{i=2}^N\int_{\B_{R}\setminus \B_{r}} g^{ia}\left(\n_t v \cdot  \n_ zv \right) z_i t_a \sqrt{|g|}  dx\nonumber\\
&=O\left(\int_{\B_{R }\setminus \B_{r}} |x|^3 |\n v|^2 dx\right)
\end{align}
By Lemma \ref{MaMetric} and Lemma \ref{DetInv}, we then have
\begin{align*}
&\sum_{ab=1}^k\int_{\B_{R}\setminus \B_{r}} [g^{ab} -\d_{ab}]\left(\de_{t_a} v \de_{t_b} v  \right)\sqrt{|g|} dx=\sum_{ab=1}^k\int_{\B_{R}\setminus \B_{r}} \biggl[-2\sum_{ij=k+1}^N z_i z_j H^i \Gamma^j_{ab}\\\
&+3 \sum_{ij=k+1}^N \sum_{c=1}^k z_i z_j \Gamma^i_{ac \Gamma^j_{bc}}+\frac{1}{3} \sum_{cd=1}^k R_{acbd}(x_0) t_c t_d+O\left(|x|^3\right)\biggl] \frac{t_a t_b}{|t|^2} |\n_t v|^2 dx.
\end{align*}
Therefore
\begin{align}\label{Exp3}
\sum_{ab=1}^k\int_{\B_{R}\setminus \B_{r}} [g^{ab} &-\d_{ab}]\left(\de_{t_a} v \de_{t_b} v  \right)\sqrt{|g|} dx=\frac{3\G^2-2H^2}{k(N-k)}\int_{\B_R\setminus \B_r} |z|^2 |\n_tv|^2 dx\nonumber\\\
&+\frac{R_g(x_0)}{3k^2} \int_{\B_R\setminus \B_r} |t|^2 |\n_t v|^2 dx+O\left(\int_{\B_R\setminus \B_r} |x|^3 |\n_t v|^2 dx\right)
\end{align}
By Lemma \ref{DetInv}, we have
\begin{align}\label{Exp4}
\int_{\B_{R}\setminus \B_{r}}|\n v|^2 (\sqrt{|g|} -1)dx&=\frac{H^2-\G^2/2}{N-k}\int_{\B_R\setminus \B_r} |z|^2 |\n v|^2 dx\\\
&-\frac{R_g(x_0)}{6k} \int_{\B_R\setminus \B_r} |t|^2 |\n v|^2 dx+O\left(\int_{\B_R\setminus \B_r} |x|^3 |\n v|^2 dx\right)
\end{align}
The result follows from \eqref{Exp00}, \eqref{Exp0}, \eqref{Exp1}, \eqref{Exp2}, \eqref{Exp3} and \eqref{Exp4}. This then ends the proof.
\end{proof}
%
We consider $\O$ a bounded domain of $\R^N$, $N\geq3$, and $\G\subset \O$ be a smooth closed submanifold of dimension $k$ with $2\leq k\leq N-2$.  For $u\in H^1_0(\O)\setminus\{0\}$, we define the ratio
\be\label{eq:def-J-u} 
J\left(u\right):=\frac{\displaystyle\int_\O |\n u |^2 dy+\int_\O h u^2 dy}{\left(\displaystyle\int_\O \rho^{-\s}_\G|u|^{2^*_\s} dy\right)^{2/2^*_\s}}.
\ee
We let 
$\eta \in \calC^\infty_c\left({Q}_{2r}\right)$ be such that
$$
0\leq \eta \leq 1 \quad \textrm{ and }\quad \eta \equiv 1 \quad \textrm{in }  \B_r .
$$
For $\e>0$, we consider  $u_\e: \O \to  \R$ given  by
\begin{equation}\label{eq:TestFunction-w}
u_\e(y):=\e^{\frac{2-N}{2}} \eta(F^{-1}_{y_0}(y)) w \left(\e^{-1} {F^{-1}_{y_0}(y)}  \right).
\end{equation}
In particular,  for every $x=(t,z)\in \R^k\times \R^{N-k}$, we have 
\begin{equation}\label{eq:TestFunction-th}
u_\e\left(F_{y_0}(x)\right):=\e^{\frac{2-N}{2}}\eta\left(x\right)\th \left( {|t|}/{\e}, {|z|}/{\e} \right).
\end{equation}
It is clear that $u_\e \in H^1_0(\O).$ Then we have the following expansion.
\begin{lemma}\label{Expansion}
For $J$  given by \eqref{eq:def-J-u}  and $u_\e$ given by \eqref{eq:TestFunction-w}, as $\e\to0$, we have 
\begin{align}\label{eq:expans-J-u-eps} 
\displaystyle J(u_\e)=S_{N,\s}&\displaystyle+\e^2\frac{H^2-3R_g(x_0)}{k(N-k)}\int_{\B_{r/\e}} |z|^2 |\n_t w|^2 dx+\e^2\frac{R_g(x_0)}{3k^2} \int_{\B_{r/\e}} |t|^2 |\n_t w|^2 dx\nonumber\\\\
 &\displaystyle+\e^2\frac{H^2+R_g(x_0)}{2(N-k)}\int_{\B_{r/\e}} |z|^2 |\n w|^2 dx-\e^2\frac{R_g(x_0)}{6k} \int_{\B_{r/\e}} |t|^2 |\n w|^2 dx\nonumber\\\\
&\displaystyle+\e^2 \frac{H^2+R_g(x_0)}{2^*_\s(N-k)}S_{N,\s} \int_{\B_{r/\e}} |z|^{2-\s} w^{2^*_\s} dx-\e^2 \frac{R_g(x_0)}{2^*_\s (3k)} S_{N,\s}\int_{\B_{r/\e}} |t|^2 |z|^{-\s} w^{2^*_\s} dx \nonumber\\\\
&+  \e^2  h(y_0) \int_{\B_{r/\e}} w^2 dx+O\left( \e^2  \int_{\B_{r/\e}} |h(F_{y_0}(\e x)- h(y_0)| w^2 dx\right)+O\left(\e^{N-2}\right)\nonumber.
\end{align}
 
\end{lemma}
\begin{proof}
To simplify the notations, we will write $F$  in the place of $F_{y_0}$.
Recalling \eqref{eq:TestFunction-w}, we   write
$$
u_\e(y)=\e^{\frac{2-N}{2}} \eta(F^{-1} (y)) W_\e(y),
$$
where $W_\e (y) =w \left(\frac{F^{-1} (y)}{\e} \right)$.
Then $
|\n u_\e |^2 =\e^{2-N} 
\left(
\eta^2 |\n W_\e|^2+\eta^2 |\n W_\e |^2+\frac{1}{2} \n W_\e^2 \cdot \n \eta^2
\right).
$
Integrating by parts, we have
\begin{align}\label{eq:expan-nabla-u-eps}
\displaystyle \int_\O |\n u_\e|^2 dy &\displaystyle=\e^{2-N} \int_{F\left({Q}_{2r}\right)} \eta^2 |\n W_\e|^2  dy+\e^{2-N} \int_{F\left({Q}_{2r}\right)\setminus F\left(\B_{r}\right)} W_\e^2 \left(|\n \eta|^2-\frac{1}{2} \D \eta^2 \right)  dy \nonumber\\
&\displaystyle=\e^{2-N} \int_{F\left({Q}_{2r}\right)} \eta^2 |\n W_\e|^2  dy-\e^{2-N} \int_{F\left({Q}_{2r}\right)\setminus F\left(\B_{r}\right)}W_\e^2 \eta \D \eta  dy  \nonumber\\
&\displaystyle=\e^{2-N} \int_{F\left({Q}_{2r}\right)} \eta^2 |\n W_\e|^2  dy+O\left(\e^{2-N} \int_{F\left({Q}_{2r}\right)\setminus F\left(\B_{r}\right)} W_\e^2  dy\right) .
\end{align}
By the  change of variable $y=\frac{F(x)}{\e}$ and \eqref{eq:TestFunction-th}, we can apply Lemma \ref{lem:to-prove}, to get
\begin{align*}
\displaystyle \int_\O |\n u_\e|^2 dy\displaystyle&= \int_{ {Q}_{r/\e} } |\n w|^2_{g_\e}\sqrt{|g_\e|} dx+O\left(\e^{2} \int_{ {Q}_{2r/\e} \setminus  \B_{r/\e} }w^2 dx+
 \int_{ {Q}_{2r/\e} \setminus  \B_{r/\e} }|\n w|^2 dx \right)\\\\
&=\int_{\R^N}|\n w|^2  dx+\e^2\frac{3\G^2-2H^2}{k(N-k)}\int_{\B_{r/\e}} |z|^2 |\n_t w|^2 dx+\e^2\frac{R_g(x_0)}{3k^2} \int_{\B_{r/\e}} |t|^2 |\n_t w|^2 dx\nonumber\\\\
 &+\e^2\frac{H^2-\G^2/2}{N-k}\int_{\B_{r/\e}} |z|^2 |\n w|^2 dx-\e^2\frac{R_g(x_0)}{6k} \int_{\B_{r/\e}} |t|^2 |\n w|^2 dx+O\left(\rho(\e)\right)\\\\
&=S_{N,\s}+\e^2\frac{3\G^2-2H^2}{k(N-k)}\int_{\B_{r/\e}} |z|^2 |\n_t w|^2 dx+\e^2\frac{R_g(x_0)}{3k^2} \int_{\B_{r/\e}} |t|^2 |\n_t w|^2 dx\nonumber\\\\
 &+\e^2\frac{H^2-\G^2/2}{N-k}\int_{\B_{r/\e}} |z|^2 |\n w|^2 dx-\e^2\frac{R_g(x_0)}{6k} \int_{\B_{r/\e}} |t|^2 |\n w|^2 dx+O\left(\rho(\e)\right),
\end{align*}
 where 
 $$
 \rho(\e)=\e^3 \int_{ \B_{r/\e}} |x|^3 |\n w|^2 dx +\e^2 \int_{\B_{2r/\e}\setminus \B_{r/\e}}  |w|^2 dx+  \int_{\R^N\setminus\B_{r/\e}}|\n w|^2 dx+  \e^2 \int_{\B_{2r/\e}\setminus \B_{r/\e}}|z|^2|\n w|^2 dx.
 $$
Using Lemma \ref{cor:dec-est-w}, we find  that  $\rho(\e)=O\left(\e^{N-2}\right)$. Therefore
\begin{align}
\displaystyle\int_\O |\n u_\e|^2 dy=S_{N,\s}&\displaystyle+\e^2\frac{3\G^2-2H^2}{k(N-k)}\int_{\B_{r/\e}} |z|^2 |\n_t w|^2 dx+\e^2\frac{R_g(x_0)}{3k^2} \int_{\B_{r/\e}} |t|^2 |\n_t w|^2 dx\nonumber\\\\
 &\displaystyle+\e^2\frac{H^2-\G^2/2}{N-k}\int_{\B_{r/\e}} |z|^2 |\n w|^2 dx-\e^2\frac{R_g(x_0)}{6k} \int_{\B_{r/\e}} |t|^2 |\n w|^2 dx+O\left(\e^{N-2}\right).
\end{align}
%
%
By the change of variable $y=\frac{F(x)}{\e}$, \eqref{DeterminantMetric}, \eqref{eq:AE} and using the fact that $\rho\left(F(x)\right)=|z|$,  we get 
\begin{align*}
&\displaystyle\int_\O \rho^{-\s}_\G |u_\e|^{2^*_\s} dy=\int_{\B_{r/\e}} |z|^{-s} w^{2^*_\s} \sqrt{|g_\e|}dx+ O\left( \int_{\B_{2r/\e}\setminus \B_{r/\e}}  |z|^{-\s} (\eta(\e x) w)^{2^*_\s} dx \right)\\
%
%
&\displaystyle= \int_{\B_{r/\e}} |z|^{-\s} w^{2^*_\s} dx+\e^2 \frac{H^2-\G^2/2}{N-k} \int_{\B_{r/\e}} |z|^{2-\s} w^{2^*_\s} dx-\e^2 \frac{R_g(x_0)}{6k} \int_{\B_{r/\e}} |t|^2 |z|^{-\s} w^{2^*_\s} dx\\
&\quad +O\left(\e^3 \int_{\B_{r/\e}} |x|^3 |z|^{-\s} w^{2^*_\s} dx+   \int_{\B_{2r/\e}\setminus \B_{r/\e}}  |z|^{-\s}  w^{2^*_\s} dx \right)\\
&\displaystyle= 1+\e^2 \frac{H^2-\G^2/2}{N-k} \int_{\B_{r/\e}} |z|^{2-\s} w^{2^*_\s} dx-\e^2 \frac{R_g(x_0)}{6k} \int_{\B_{r/\e}} |t|^2 |z|^{-\s} w^{2^*_\s} dx\\
&\quad +O\left(\e^3 \int_{\B_{r/\e}} |x|^3 |z|^{-\s} w^{2^*_\s} dx + \int_{\R^N \setminus \B_{r/\e}} |z|^{-\s} w^{2^*_\s} dx+ \int_{\B_{2r/\e}\setminus \B_{r/\e}}  |z|^{-\s}  w^{2^*_\s} dx\right).
\end{align*}
Using  \eqref{DecayEstimates111}, we have
$$
\e^3 \int_{\B_{r/\e}} |x|^3 |z|^{-\s} w^{2^*_\s} dx+\int_{\R^N \setminus \B_{r/\e}} |z|^{-\s} w^{2^*_\s} dx+\int_{\B_{2r/\e}\setminus \B_{r/\e}}  |z|^{-\s}  w^{2^*_\s} dx=O\left(\e^{N-\s}\right).
$$
Hence by Taylor expanding, we get
$$
\left(\int_\O \rho^{-\s}_\G |u_\e|^{2^*_\s} dx\right)^{2/2^*_\s} =1+\e^2 \frac{2H^2-\G^2}{2^*_\s(N-k)} \int_{\B_{r/\e}} |z|^{2-\s} w^{2^*_\s} dx-\e^2 \frac{R_g(x_0)}{2^*_\s (3k)} \int_{\B_{r/\e}} |t|^2 |z|^{-\s} w^{2^*_\s} dx+O\left(\e^{N-\s}\right).
$$
Finally, by \eqref{eq:expan-nabla-u-eps}, we conclude that
\begin{align*} 
\displaystyle J(u_\e)=S_{N,\s}&\displaystyle+\e^2\frac{3\G^2-2H^2}{k(N-k)}\int_{\B_{r/\e}} |z|^2 |\n_t w|^2 dx+\e^2\frac{R_g(x_0)}{3k^2} \int_{\B_{r/\e}} |t|^2 |\n_t w|^2 dx\nonumber\\\\
 &\displaystyle+\e^2\frac{H^2-\G^2/2}{N-k}\int_{\B_{r/\e}} |z|^2 |\n w|^2 dx-\e^2\frac{R_g(x_0)}{6k} \int_{\B_{r/\e}} |t|^2 |\n w|^2 dx\\\\
&\displaystyle+\e^2 \frac{2H^2-\G^2}{2^*_\s(N-k)}S_{N,\s} \int_{\B_{r/\e}} |z|^{2-\s} w^{2^*_\s} dx-\e^2 \frac{R_g(x_0)}{2^*_\s (3k)} S_{N,\s}\int_{\B_{r/\e}} |t|^2 |z|^{-\s} w^{2^*_\s} dx 
\nonumber\\\\
&+ \e^2  h(y_0) \int_{\B_{r/\e}} w^2 dx+O\left( \e^2  \int_{\B_{r/\e}} |h(F_{y_0}(\e x)- h(y_0)| w^2 dx\right)+O\left(\e^{N-2}\right).
\end{align*}
We thus get the desired result by using the Gauss equation $\G^2=H^2-R_g(x_0)$, see [\cite{Gray}, Chapter 4].
\end{proof}
%
%
%
%
%
%
%
%
%
\begin{proposition}\label{Proposition2}
  For $N\geq 5$, we define
$$
A_{N,\s}=\displaystyle\frac{1}{k(N-k)}\int_{\B_{r/\e}} |z|^2 |\n_t w|^2 dx+\frac{1}{2(N-k)}\int_{\B_{r/\e}} |z|^2 |\n w|^2 dx\displaystyle+\frac{1}{2^*_\s(N-k)}S_{N,\s} \int_{\B_{r/\e}} |z|^{2-\s} w^{2^*_\s} dx,
$$ 
\begin{align*}
B_{N,\s}=&-\displaystyle\frac{3}{k(N-k)}\int_{\B_{r/\e}} |z|^2 |\n_t w|^2 dx+\frac{1}{3k^2} \int_{\B_{r/\e}} |t|^2 |\n_t w|^2 dx+\frac{1}{2(N-k)}\int_{\B_{r/\e}} |z|^2 |\n w|^2 dx\nonumber\\\\
&\displaystyle-\frac{1}{6k} \int_{\B_{r/\e}} |t|^2 |\n w|^2 dx+\frac{S_{N,\s}}{2^*_\s(N-k)} \int_{\B_{r/\e}} |z|^{2-\s} w^{2^*_\s} dx-\frac{S_{N,\s}}{2^*_\s (3k)}\int_{\B_{r/\e}} |t|^2 |z|^{-\s} w^{2^*_\s} dx
\end{align*}
and
$$
C_{N,\s}=\int_{\R^N} |\n w|^2 dx.
$$
Assume that, for some $y_0 \in \G$, we have
$$
\begin{cases}
\displaystyle \frac{A_{N,\s}H^2 +B_{N,\s} R_g(y_0)}{{C_{N,\s}}}+h(y_0)<0 & \qquad \textrm{for } N \geq 5\\\\
\displaystyle  A_k H^2(y_0)+B_k R_g(y_0)+h(y_0)<0 & \qquad \textrm{for } N=4,
\end{cases}
$$
Then
$$
\mu_{h,\s}\left(\O,\G \right) < S_{N,\s}.
$$
\end{proposition}
\begin{proof}
We claim that
 \begin{align}\label{eq:z-sqrt-w-2-star}
S_{N,\s} \int_{Q_{r/\e}}      |z|^{2-\s}  w^{2^*_\s} dx&= \int_{Q_{r/\e}}  |z|^2 |\n w|^2 dx-(N-k)\int_{Q_{r/\e}} w^2        dx +O(\e^{N-2})\\
S_{N,\s} \int_{Q_{r/\e}}      |t|^2 |z|^{-\s}  w^{2^*_\s} dx&= \int_{Q_{r/\e}}  |t|^2 |\n w|^2 dx-k\int_{Q_{r/\e}} w^2        dx +O(\e^{N-2})
\end{align} 
To prove this claim, we let $\eta_\e(x)=\eta(\e x)$. 
We multiply \eqref{ExpoEA} by $|z|^2\eta_\e w$ and integrate by parts to get
\begin{align*}
\displaystyle S_{N,\s}& \int_{Q_{2r/\e}}    \eta_\e  |z|^{2-\s}  w^{2^*_\s} dx\displaystyle=\int_{Q_{2r/\e}} \n w \cdot \n \left(\eta_\e  |z|^2 w\right) dx\\
&\displaystyle= \int_{Q_{2r/\e}} \eta_\e  |z|^2 |\n w|^2 dx+\frac{1}{2} \int_{Q_{2r/\e}} \n w^2 \cdot \n \left(|z|^2\eta_\e \right) dx \int_{Q_{2r/\e}} \eta_\e  |z|^2 |\n w|^2 dx-\frac{1}{2} \int_{Q_{2r/\e}} w^2 \D\left(|z|^2\eta_\e \right) dx\\
&\displaystyle=\int_{Q_{2r/\e}} \eta_\e  |z|^2 |\n w|^2 dx-(N-1)\int_{Q_{2r/\e}} w^2 \eta_\e      dx =\displaystyle\quad-\frac{1}{2} \int_{Q_{2r/\e}\setminus Q_{r/\e}} w^2 (|z|^2\D\eta_\e+ 4\n \eta_\e\cdot z) dx.
\end{align*} 
We then deduce that 
\begin{align*}
S_{N,\s} \int_{Q_{r/\e}}      |z|^{2-\s}  w^{2^*_\s} dx&= \int_{Q_{r/\e}}  |z|^2 |\n w|^2 dx-(N-1)\int_{Q_{r/\e}} w^2        dx\\
&+O\left( \int_{Q_{2r/\e}\setminus Q_{r/\e}}      |z|^{2-\s}  w^{2^*_\s} dx+  \int_{Q_{2r/\e}\setminus  Q_{r/\e}}  |z|^2 |\n w|^2 dx+ \int_{Q_{2r/\e}\setminus  Q_{r/\e}}     w^2 dx     \right)\\
&+ O\left(  \e \int_{Q_{2r/\e}\setminus  Q_{r/\e}}  |z|  |\n w| dx+ \e^2 \int_{Q_{2r/\e}\setminus  Q_{r/\e}}   |z|^2   w^2 dx     \right).
\end{align*} 
Thanks to Lemma \ref{cor:dec-est-w}, we  get the first equation of \eqref{eq:z-sqrt-w-2-star} as claimed. For the second one we multiply \eqref{ExpoEA} by $|t|^2\eta_\e w$ and integrate by parts as in the first one.

Next, by the continuity of  $h$, for $\delta>0$, we can find  $r_\d>0$ such that
\be \label{eq:cont-h-eff}
|h(y)-h(y_0)|<\d \qquad \textrm{for ever $y\in F\left(\B_{r_\d}\right)$ }.
\ee

\noindent
\textbf{Case $N\geq 5.$}\\
\noindent
Using  \eqref{eq:z-sqrt-w-2-star} and \eqref{eq:cont-h-eff}  in   \eqref{eq:expans-J-u-eps},  we obtain, for every $r\in (0,r_\d)$
\begin{align}\label{eq:expans-J-u-eps} 
\displaystyle J(u_\e)=S_{N,\s}&\displaystyle+\e^2\frac{H^2-3R_g(x_0)}{k(N-k)}\int_{\R^N} |z|^2 |\n_t w|^2 dx+\e^2\frac{R_g(x_0)}{3k^2} \int_{\R^N} |t|^2 |\n_t w|^2 dx\nonumber\\\\
 &\displaystyle+\e^2\frac{H^2+R_g(x_0)}{2(N-k)}\int_{\R^N} |z|^2 |\n w|^2 dx-\e^2\frac{R_g(x_0)}{6k} \int_{\B_{r/\e}} |t|^2 |\n w|^2 dx\nonumber\\\\
&\displaystyle+\e^2 \frac{H^2+R_g(x_0)}{2^*_\s(N-k)}S_{N,\s} \int_{\R^N} |z|^{2-\s} w^{2^*_\s} dx-\e^2 \frac{R_g(x_0)}{2^*_\s (3k)} S_{N,\s}\int_{\R^N} |t|^2 |z|^{-\s} w^{2^*_\s} dx\nonumber\\\\
&+  \e^2  h(y_0) \int_{\R^N} w^2 dx+O\left( \e^2 \d^2 \int_{\R^N}   w^2 dx\right)+O\left(\e^{N-2}\right),
\end{align}
where we have used Lemma \ref{cor:dec-est-w} to get the estimates
$$
\int_{  \R^N\setminus \B_{r/\e}} |z|^2 |\n w|^2 dx+   \int_{\R^N\setminus \B_{r/\e}} w^2 dx=O(\e).
$$
It follows that, for every $r\in (0,r_\d)$,
\begin{align*}
J\left(u_\e\right)=  S_{N,\s}+\e^2\left\lbrace A_{N,\s} H^2(y_0)+B_{N,\s} R_g(y_0)+C_{N,\s} h(y_0)\right\rbrace+O(\d\e^2   B_{N,\s})+O\left(\e^{3}\right).
\end{align*}
Suppose now that 
$$
A_{N,\s} H^2(y_0)+B_{N,\s} R_g(y_0)+C_{N,\s} h(y_0)<0
$$
We can thus    choose respectively  $\delta>0$ small and    $\e>0$ small so that  $ J(u_\e)< S_{N,\s}$. Hence we get
$$
\mu_{h,\s}\left(\O ,\G\right)< S_{N,\s}.
$$
\noindent
\textbf{Case $N=4.$}\\
\noindent
From \eqref{eq:expans-J-u-eps} and \eqref{eq:cont-h-eff}, we estimate, for every $r\in (0,r_\d)$
\begin{align*}\label{eq:expans-J-u-eps} 
\displaystyle J(u_\e)\leq S_{N,\s}&\displaystyle+\e^2\frac{|H^2-3R_g(x_0)|}{k(N-k)}\int_{\B_{r/\e}} |z|^2 |\n w|^2 dx+\e^2\frac{|R_g(x_0)|}{3k^2} \int_{\B_{r/\e}} |t|^2 |\n w|^2 dx\nonumber\\\\
 &\displaystyle+\e^2\frac{H^2+R_g(x_0)}{2(N-k)}\int_{\B_{r/\e}} |z|^2 |\n w|^2 dx-\e^2\frac{R_g(x_0)}{6k} \int_{\B_{r/\e}} |t|^2 |\n w|^2 dx\nonumber\\\\
&\displaystyle+\e^2 \frac{H^2+R_g(x_0)}{2^*_\s(N-k)}S_{N,\s} \int_{\B_{r/\e}} |z|^{2-\s} w^{2^*_\s} dx-\e^2 \frac{R_g(x_0)}{2^*_\s (3k)} S_{N,\s}\int_{\B_{r/\e}} |t|^2 |z|^{-\s} w^{2^*_\s} dx \nonumber\\\\
&+  \e^2  h(y_0) \int_{\B_{r/\e}} w^2 dx+O\left( \e^2  \d\int_{\B_{r/\e}} w^2 dx\right)+O\left(\e^{N-2}\right)\nonumber.
\end{align*}
Further since, by \eqref{DecayEstimates111},  
$$
\int_{\B_{r/\e}} |z|^{2-\s} w^{2^*_\s} dx=O(1),
$$
then by \eqref{eq:z-sqrt-w-2-star}, we get
\begin{align*}
\displaystyle J(u_\e)\leq S_{N,\s}&\displaystyle+\e^2\frac{|H^2-3R_g(x_0)|}{k}\int_{\B_{r/\e}} w^2 dx+\e^2\frac{|R_g(x_0)|}{3k} \int_{\B_{r/\e}} w^2dx\nonumber\\\\
 &\displaystyle+\e^2\frac{H^2+R_g(x_0)}{2}\int_{\B_{r/\e}} w^2 dx-\e^2\frac{R_g(x_0)}{6} \int_{\B_{r/\e}} w^2 dx\nonumber\\\\
&+  \e^2  h(y_0) \int_{\B_{r/\e}} w^2 dx+O\left( \e^2  \d\int_{\B_{r/\e}} w^2 dx\right)+O\left(\e^{N-2}\right)\nonumber.
\end{align*}
Therefore
\begin{align*}
\displaystyle J(u_\e)\leq S_{N,\s}\displaystyle&+\e^2\left[ \frac{|H^2(y_0)-3R_g(y_0)|}{k}+\frac{|R_g(y_0)|}{3k}+ \frac{H^2(y_0)}{2}+\frac{R_g(y_0)}{3}+h(y_0)\right]\int_{\B_{r/\e}} w^2 dx\\\
&+O\left( \e^2  \d\int_{\B_{r/\e}} w^2 dx\right)+O\left(\e^{N-2}\right).
\end{align*}
Thus 
\begin{align*}
\displaystyle J(u_\e)\leq S_{N,\s}\displaystyle&+\e^2\left[ \frac{|H^2(y_0)-3R_g(y_0)|}{k}+\frac{|R_g(y_0)|}{3k}+ \frac{H^2(y_0)}{2}+\frac{R_g(y_0)}{3}+h(y_0)\right]\int_{\B_{r/\e}} w^2 dx\\\
&+O\left( \e^2  \d\int_{\B_{r/\e}} w^2 dx\right)+C\e^2,
\end{align*}
for some positive constant $C$ independent on $\e$.
By \eqref{DecayEstimates111}, we have that
$$
\begin{array}{ll}
\displaystyle \int_{\B_{r/\e}} \frac{C_1^2}{1+|x|^2} dx \leq \int_{\B_{r/\e}} w^2 dx \leq \int_{\B_{r/\e}} \frac{C_2^2}{1+|x|^2} dx,
\end{array}
$$
so that 
\begin{equation}\label{EtAlors0}
\begin{array}{ll}
\displaystyle \int_{B_{\R^4}(0,r/\e)} \frac{C_1^2}{\left(1+|x|^2\right)^2} dx \leq \int_{\B_{r/\e}} w^2 dx \leq \int_{B_{\R^4}(0,2r/\e)} \frac{C_2^2}{\left(1+|x|^2\right)^2} dx.
\end{array}
\end{equation}
Using polar coordinates and  a change of variable, for $R>0$, we have
$$
\begin{array}{ll}
\displaystyle \int_{B_{\R^4}(0,R)} \frac{dx}{\left(1+|x|^2\right)^2} dx
&\displaystyle= |S^3| \int_0^{R} \frac{t^3}{\left(1+t^2\right)^2} dt\\
&\displaystyle= |S^3| \int_0^{\sqrt{R}} \frac{s}{2\left(1+s\right)^2} ds\\
&\displaystyle=\frac{|S^3|}{2}\left(\log\left(1+\sqrt{R}\right)-\frac{\sqrt{R}}{1+\sqrt{R}}\right).
\end{array}
$$
Therefore, there exist numerical constants $c,\ov c>0$ such that for every  $\e>0$ small, we have  
\begin{equation}\label{EtAlors1}
 c   |\log  \e |\leq   \int_{\B_{r/\e}} w^2 dx \leq \ov c   |\log  \e |.
\end{equation}
Now we assume that  
$$
\frac{|H^2(y_0)-3R_g(y_0)|}{k}+\frac{|R_g(y_0)|}{3k}+ \frac{H^2(y_0)}{2}+\frac{R_g(y_0)}{3}+h(y_0)<0.
$$ 
Therefore by Lemma \ref{Expansion}  and \eqref{EtAlors1}, we get  
$$
\displaystyle J(u_\e)\leq S_{N,\s}\displaystyle+\ov c \left[\frac{|H^2(y_0)-3R_g(y_0)|}{k}+\frac{|R_g(y_0)|}{3k}+ \frac{H^2(y_0)}{2}+\frac{R_g(y_0)}{3}+h(y_0)\right]\e^2 |\log \e|+ \ov c \d\e^2  |\log \e|+ C \e^{2} .
$$
 Then choosing $\d>0$ small and     $\e$ small, respectively, we deduce that  $\mu_{h,\s}\left(\O,\G\right)\leq J(u_\e)< S_{4,\s}$.
This ends the proof of the proposition.
\end{proof}
\begin{proof}[Proof of Theorem \ref{MainResult} (completed)]
We know that when $\mu_{h,\s}\left(\O, \G\right) < S_{N,\s}$ then $\mu_{h,\s}\left(\O, \G\right)$ is achieved by a positive function $u$, see Proposition \ref{Pro3333} above. Therefore by Proposition \ref{Proposition2},  we get the result with $C^1_{N,\s}=\frac{A_{N,\s}}{C_{N,\s}}$ and $C^2_{N,\s}=\frac{B_{N,\s}}{C_{N,\s}}$ for $N\geq 5$. When $N=4$, we get $C^1_{4,\s}$ and $C^2_{4,\s}$ depending on the signs of $H^2(y_0)-3R_g(y_0)$ and $R_g(y_0)$. They are given by
\begin{align}
\displaystyle C^1_{4,\s}=0& \quad\textrm{and}\quad C^2_{4,\s}= \frac{1}{3}-\frac{10}{3k} \quad &\textrm{when}\quad H^2\leq R_g(y_0) \leq 0.\nonumber\\\
\displaystyle C^1_{4,\s}=\frac{1}{2}+\frac{1}{k}& \quad\textrm{and}\quad C^2_{4,\s}= \frac{1}{2}-\frac{8}{3k} \quad &\textrm{when}\quad H^2\geq 3 R_g(y_0) \geq 0.\nonumber\\\
\displaystyle C^1_{4,\s}=\frac{1}{2}+\frac{1}{k}& \quad\textrm{and}\quad C^2_{4,\s}= \frac{1}{2}-\frac{10}{3k} \quad &\textrm{when} \quad R_g(y_0) \leq 0.\label{Options}\\\
\displaystyle C^1_{4,\s}=\frac{1}{2}-\frac{1}{k}& \quad\textrm{and}\quad C^2_{4,\s}= \frac{1}{2}-\frac{10}{3k} \quad &\textrm{when} \quad H^2-3R_g(y_0) \leq 0 \quad\textrm{and} \quad R_g(y_0) \geq 0.\nonumber
\end{align}
\end{proof}
\bigskip
\noindent
\textbf{Acknowledgement:}
I wish to thanks my supervisor Mouhamed Moustapha Fall for useful discussions and remarks.  This work is  supported by the German Academic Exchange Service (DAAD).

\end{document}